\documentclass[11pt]{amsart}
\usepackage{amsmath, palatino, mathpazo, amsfonts, amssymb, mathtools, tikz-cd}
\usepackage[mathscr]{eucal}
\usepackage[all]{xy}
\usepackage{datetime}

\usepackage{amsthm}
\usepackage {tikz}
\usepackage {tgbonum}                 
\usepackage {amsmath}                   
\usepackage {amssymb}                   
\usepackage {bibnames}                  
\usepackage {cite}                      
\usepackage {url}                       
\usepackage {varioref}                  

\usepackage{tabularx}
\usepackage{array}

\newtheorem{theorem}{Theorem}[section]
\newtheorem{lemma}[theorem]{Lemma}
\newtheorem{proposition}[theorem]{Proposition}
\newtheorem{corollary}[theorem]{Corollary}
\newtheorem{conjecture}[theorem]{Conjecture}

\theoremstyle{remark}
\newtheorem{remark}[theorem]{Remark}
\newtheorem{example}[theorem]{Example}

\numberwithin{equation}{section}

\newcommand{\op}{\operatorname}

\newcommand{\M}{\overline{\mathcal{M}}}

\newcommand{\vir}{{\rm vir}}
\newcommand{\ev}{{\rm ev}}

\newcommand{\td}{{\rm td}}
\newcommand{\ch}{{\rm ch}}

\def\O{\mathcal{O}}

\newcommand{\GW}{{\mathrm{GW}}}
\newcommand{\GV}{{\mathrm{GV}}}

\newcommand{\e}{\epsilon}
\newcommand{\tw}{\mathrm{tw}}

\title[QK $=$ GV]{Gopakumar--Vafa invariants $=$ quantum $K$-invariants \\ on Calabi--Yau threefolds}
	
	\author{Y.-C.~Chou}
	\email{bensonchou@gate.sinica.edu.tw, chou@math.utah.edu}
	
	\author{Y.-P.~Lee}
	\email{yplee@math.utah.edu, ypleemath@gate.sinica.edu.tw}
	
\address{Institute of Mathematics, Academia Sinica, Taipei 10617, Taiwan, and
Department of Mathematics, University of Utah, 	Salt Lake City, Utah 84112-0090, U.S.A.}

\date{\today}

\newenvironment{dedication}
  {
   \itshape             
   \raggedleft          
  }
  {\par 
    \vspace{30pt}
  }

\begin{document}
\begin{dedication}
In memory of Bumsig Kim
\end{dedication}

\maketitle

\begin{abstract}
The main purpose of this article is to discuss a project relating Gopakumar--Vafa invariants to quantum $K$-invariants on Calabi--Yau threefolds. Results in genus zero, including recent and forthcoming works, are reported. 
\end{abstract}



\section{Introduction}
This is an expository article based on Part I \cite{Chou_Lee_2022} and Part II \cite{Chou_Lee_2023} of our ongoing project about quantum $K$-theory of Calabi--Yau threefolds (CY3). In this project, we hope to explore various special properties of \emph{quantum $K$-theory on Calabi--Yau threefolds}, which have not yet been extensively studied in the mathematical literature. We chose to start with its relation with the Gopakumar--Vafa invariants partly because we feel that the \emph{integrality} of quantum $K$-theory has not received much attention so far.

Gromov--Witten theory of Calabi--Yau threefolds enjoys various beautiful properties, among them the \emph{Kodaira--Spencer theory of gravity} and \emph{holomorphic anomaly equation} of Bershadsky, Cecotti, Ooguri and Vafa, the appearance of (quasi-)modular forms etc.. Some of these remarkable phenomena should have counterparts in $K$-theory. We hope to explore these properties in the context of quantum $K$-theory in the future.


\subsection{Gopakumar--Vafa invariants and Quantum $K$-theory}
Among the \emph{integral} (virtual) enumerative invariants on Calabi--Yau threefolds (CY3), two will be the foci of this paper. The first, called the \emph{Gopakumar--Vafa invariants} (GV) \cite{Gopakumar_Vafa_1, Gopakumar_Vafa_2}, was introduced 
in theoretical physics as ``new topological invariants on \emph{Calabi--Yau threefolds}'', counting the
``numbers of BPS states''. There have been various attempts at giving these BPS invariants rigorous mathematical definitions. At the moment, there are still unresolved issues of these definitions. We refer the readers to \cite{Maulik_Toda_2018} and references therein for mathematical definitions of various degrees of generalities of these invariants. Gopakumar and Vafa also argue that these BPS invariants and the Gromov--Witten invariants are intimately related by the following formula 
\begin{equation} \label{e:1.1}
\begin{split}
 &\sum _{g=0}^{\infty }~\sum _{\beta \in H_{2}(M,\mathbb {Z} )}{\GW}_{g,\beta} q^{\beta }\lambda ^{2g-2} \\
 = & \sum _{g=0}^{\infty }~\sum _{k=1}^{\infty }~\sum _{\beta \in H_{2}(M,\mathbb {Z} )}{\GV}_{g,\beta}{\frac {1}{k}}\left(2\sin \left({\frac {k\lambda }{2}}\right)\right)^{2g-2}q^{k\beta } ,
\end{split}
\end{equation}
which all viable mathematical definitions must satisfy. In fact, this formula is ultimately the most important test for any geometric definition. The detailed definitions and discussions can be found in Section~\ref{s:2.2}. In this paper, we will use the above (invertible) relation as the \emph{definition} of the Gopakumar--Vafa invariants (in terms of Gromov--Witten invariants). As defined, it is not at all obvious these invariants are integral. The integrality of this ad hoc definition was proven by E.~Ionel and T.~Parker using symplectic techniques in a remarkable work \cite{Ionel_Parker_2018} .

The second is the \emph{quantum $K$-invaraints} (QK) \cite{Givental_2000, Lee_2004}, a $K$-theoretic variant of the (cohomological) Gromov--Witten invariants (GW). The definition and references for quantum $K$-theory will be recalled in Section~\ref{section_KGW_J}. Whereas Gromov--Witten invariants are \emph{rational} numbers, quantum $K$-invariants, counting alternating sums of ranks of sheaf cohomology (Euler characteristic), produces \emph{integral} invariants by definition. Therefore, any relations established between the quantum $K$-invariants and Gopakumar--Vafa invariants will automatically produce the \emph{integrality for the GV}.

\subsection{GV $=$ QK on CY3?}
The question is: Are there any relations between the Gopakumar--Vafa invariants and quantum $K$-invariants on Calabi--Yau threefolds? 
We think this is the case. In fact, we think that one should be able to \emph{define} the Gopakumar--Vafa invariants (GV) in terms of the quantum $K$-invariants, similar to \eqref{e:1.1}.  While this definition would still be \emph{ad hoc}, it would have the benefit of being integral by definition. It would also give an alternative proof of Ionel--Parker's integrality theorem within the algebraic category.

Why do we believe such relationship exist? On the conceptual level, there is a clear link between GV and QK through GW. On the one hand, as pointed out above, Gopakumar--Vafa invariants and Gromov--Witten invariants are reconstructible from each other by the Gopakumar--Vafa formula \eqref{e:1.1}. On the other hand, Givental and his collaborators has furnished a clear link between QK and GW via a virtual orbifold Hirzebruch--Riemann--Roch theorem \cite{Givental_Tonita_2011, Givental_PEK}.

The close relationship between the quantum $K$-theory and quantum cohomology was understood since the early phase of quantum $K$-theory. In fact, many $K$-theoretic results were inspired by their cohomological counterparts. Many formulas and results were first guessed based on their cohomological counterparts. See, e.g., \cite{Givental_Lee_2003, Lee_2004}. In \cite{Givental_Tonita_2011} and subsequent works \cite{Givental_PEK}, A.~Givental and his collaborators completely characterized genus zero quantum $K$-theory in terms of quantum cohomology. Subsequently, Givental generalizes these results to permutation equivariant setting and to higher genera \cite{Givental_PEK}.

Therefore, there is a relation between quantum $K$-invariants and Gopakumar--Vafa invariants for Calabi--Yau threefolds via Gromov--Witten invariants. In this series of papers, we seek to explore these links and to ``codify'' the relations between QK and GV. Our starting point is the conjecture by H.~Jockers, P.~Mayr \cite{Jockers_Mayr_2019} and S.~Garoufalidis, E.~Scheidegger \cite{Garoufalidis_Scheidegger_2022} in genus zero. We will formulate a precise version of the conjecture for all Calabi--Yau threefolds in genus zero based on their works. 

The Gopakumar--Vafa formula \eqref{e:1.1}, together with results in Section~\ref{s:3.1}, can be interpreted as the following statement: \emph{in genus zero}, the collection of all Gopakumar--Vafa invariants of a fixed Calabi--Yau threefold contains exactly the same information as the collection of all Gromov--Witten invariants. 
We hope that in the future to study the higher genus counterpart. We will offer some heuristic arguments why this might be possible if higher genus computations can be done for $(-1,-1)$ curves, the so-called multiple cover formula.

\subsection*{Acknowledgements}
We wish to thank A.~Givental, R.~Pandharipande, H.~Tseng and Y.~Wen for discussions about this work.
The research is partially supported by the Simons Foundation, the NSTC,  University of Utah and Academia Sinica.

\section{GW, GV and QK} \label{s:2}

\subsection{Gromov--Witten invariants}
Let $X$ be a smooth complex projective variety, and $\M_{g,n}(X,\beta)$ be the M.~Kontsevich’s moduli space of $n$-pointed, genus $g$, degree $\beta$ stable maps. 
Given $i\in \{ 1,\dots,n\}$, there is an evaluation map 
\[
\begin{split}
\ev_i : \M_{g,n}(X,\beta) & \rightarrow X
\\
[f: (C;x_1,\dots,x_n)] & \mapsto f(x_i),
\end{split}
\]
and a line bundle $L_i := x_i^* w_{\mathcal{C}/\M}$ on $\M_{g,n}(X,\beta)$, where $w_{\mathcal{C}/\M}$ is the relative dualizing sheaf of the universal curve $\mathcal{C} \rightarrow \M_{g,n}(X,\beta)$ and $x_i: \M_{g,n}(X,\beta) \rightarrow \mathcal{C}$ is the $i$-th mark point. 

(Cohomological) Gromov-Witten invariants of $X$ are defined to be
\[
\langle \tau_{k_1}(\phi_1)\dots \tau_{k_n}(\phi_n) \rangle^{X, {H}}_{g,n,\beta} :=  \pi^H_* \left( \cup_{i=1}^n \ev_i^*(\phi_i) c_1(L_i)^{k_i} \cap [\M_{g,n}(X,\beta)]^{\vir} \right) \in \mathbb{Q},
\]
where 
\[
 \pi: \M_{g,n}(X,\beta) \to pt := \op{Spec}(\mathbb{C})
\]
is the structural map and $\pi^H_*$ is the (cohomological) pushforward to the point.
Here $\phi_1,\dots,\phi_n \in H(X)$, $k_1,\dots, k_n \in \mathbb{Z}_{\geq 0}$, and $[\M_{g,n}(X,\beta)]^{\vir}$ are the (cohomological) virtual fundamental classes. 

The genus $g$ Gromov--Witten invariants of $X$ can be encoded in a generating function, called genus-$g$ descendant potential
\[
F^{{H}}_g(t) = \sum_{n\geq 0}\sum_{\beta} \frac{Q^{\beta}}{n!} \langle  t(L),\dots,t(L)  \rangle^{X,{H}}_{g,n,\beta}.
\]
Here the sum is over all curve class $\beta \in H_2(X)_{\geq 0}$ and $Q^{\beta}$ are formal variables, called the Novikov variables, which keep track of the curve classes. $t(q)$ stands for any polynomial of one variable with coefficients in $H(X)$.
That is 
\[
t(q) = \sum_{k\in \mathbb{Z}_{\ge 0}} \sum_{\alpha=1}^N t_{k}^{\alpha} \phi_{\alpha} q^k
\] 
with $\{\phi_{\alpha} \}_{\alpha=1}^N$ a basis of $H(X)$.

\subsection{Gopakumar--Vafa invariants} \label{s:2.2}
In theoretical physics, R.~Gopakumar and C.~Vafa in \cite{Gopakumar_Vafa_1, Gopakumar_Vafa_2} introduced new topological invariants on \emph{Calabi--Yau threefolds} (CY3) $X$, which are now commonly called \emph{Gopakumar--Vafa invariants}.
These invariants represent the counts of ``numbers of BPS states'' on $X$. Unlike the Gromov--Witten invariants, which are defined for any symplectic manifolds or orbifolds, Gopakumar--Vafa invariants only make sense for Calabi--Yau threefolds (or their variants).

The virtual dimensions for moduli spaces (of stable maps) to CY3 are always equal to the number of marked points. That is,
\[
 \op{vdim} \M_{g,n} (X, \beta) = n.
\]
By general properties of the moduli spaces, more precisely the string equation, the divisor equation and the dilaton equation, all Gromov--Witten invariants can be easily reconstructed from $0$-pointed invariants
\[
 {\GW}_{g, \beta} := \pi^H_* \left( [\M_{g,0}(X,\beta)]^{\vir} \right) = \int_{[\M_{g,0}(X,\beta)]^{\vir} } 1 .
\]
In fact, a \emph{closed formula} was obtained in \cite{Fan_Lee_2019} in terms of generating functions. We will therefore focus on $0$-pointed invariants ${\GW}_{g, \beta}$.

There have been various attempts at defining Gopakumar--Vafa invariants mathematically. We refer the readers to \cite{Maulik_Toda_2018} and references therein. If one is interested in defining Gopakumar--Vafa invariants to allow for insertions, the proposed generalizations are expected to be compatible with the dilaton and divisor equations. Therefore, the counting of the BPS states is again reduced to similarly defined ${\GV}_{g, \beta}$.

A remarkable relation between GV and GW in \cite{Gopakumar_Vafa_1, Gopakumar_Vafa_2} can be expressed in terms of generating functions:
\begin{equation} \label{e:1.1'}
\begin{split}
 &\sum _{g=0}^{\infty }~\sum _{\beta \in H_{2}(M,\mathbb {Z} )}{\GW}_{g,\beta} q^{\beta }\lambda ^{2g-2} \\
 = & \sum _{g=0}^{\infty }~\sum _{k=1}^{\infty }~\sum _{\beta \in H_{2}(M,\mathbb {Z} )}{\GV}_{g,\beta}{\frac {1}{k}}\left(2\sin \left({\frac {k\lambda }{2}}\right)\right)^{2g-2}q^{k\beta } .
\end{split}
\end{equation}
We note that this formula gives an invertible relation between GV and GW, filtered by genus. Namely, one can obtain all $\{\GV_{g, \beta}\}_{g \leq g_0, \beta}$ from $\{\GW_{g, \beta}\}_{g \leq g_0, \beta}$ and vice versa. In this paper, we use this formula to \emph{define} the Gopakumar--Vafa invariants.

\begin{example}
For the quintic CY3, the above relation can be written as
\begin{equation} \label{e:1.2}
\begin{split}
    {\GW}_{g=0, \beta= d [\op{line}]} =: {\GW}_{d} & = \sum_{e|d} \frac{1}{e^3} {\GV}_{d/e}, \\ 
    {\GV}_{g=0, \beta= d [\op{line}]} =:     {\GV}_{d} & := \sum_{e|d} \frac{\mu(e)}{e^3} {\GW}_{d/e},
\end{split}
\end{equation}
where we have used the M\"obius inversion and $\mu(e)$ is the M\"obius function. For our purpose, we will \emph{use \eqref{e:1.2} as the definition of ${\GV}_{d}$}. The first 4 terms are listed for readers' convenience.
\begin{table}[ht]
\begin{tabular}{|l|l|l|l|l|}
\hline
$d$ & 1 & 2 & 3 & 4   \\ \hline
$\GW_d$ & 2875 & 4876875/8 & 8564575000/27 & 15517926796875/64  \\ \hline
$\GV_d$ & 2875  & 609250 & 317206375 & 242467530000   \\ \hline
\end{tabular}
\end{table}
\end{example}

This \emph{ad hoc} definition has among other things one difficulty. Namely, the Gopakumar--Vafa invariants are to be intrinsically \emph{integers}, while the Gromov--Witten invariants are generally \emph{rational} numbers, as the above table demonstrates. Fortunately, the integrality of this definition has been shown in \cite{Ionel_Parker_2018}.

There is a variant of the Gromov--Witten theory which also produces \emph{integral} invariants, namely, the \emph{quantum $K$-theory} \cite{Givental_2000, Lee_2004}.
This leads to the possibility of relating quantum $K$-invariants with Gopakumar--Vafa invariants for Calabi--Yau threefolds.

\subsection{Quantum $K$-invariants} \label{section_KGW_J}

The formulation of \emph{quantum $K$-theory} is similar to that of quantum cohomology, i.e., Gromov--Witten theory.
The \emph{quantum $K$-invariants}, or $K$-theoretic Gromov-Witten invariants, of $X$ are defined to be
\[
\langle \tau_{d_1}(\Phi_1)\dots \tau_{d_n}(\Phi_n) \rangle^{X, {K}}_{g,n,\beta} :=  \chi \Big(  \M_{g,n}(X,\beta) ; \Big(  \otimes_{i=1}^n \ev_i^*(\Phi_i) L_i^{d_i}  \Big) \otimes \O^{\vir}  \Big) \in \mathbb{Z}.
\]
Here $\Phi_1,\dots,\Phi_n \in K^0(X)$, $d_1,\dots,d_n \in \mathbb{Z}$, and $\O^{\vir}$ is the virtual structure sheaf on $\M_{g,n}(X,\beta)$ \cite{Lee_2004}. As in Gromov--Witten theory, all genus $g$ invariants can be packed into a formal power series, called genus-$g$ descendant potential of $X$:
\[
F^{{K}}_g(t) = \sum_{n\geq 0}\sum_{\beta} \frac{Q^{\beta}}{n!} \langle  t(L),\dots,t(L)  \rangle^{X,{K}}_{g,n,\beta}.
\]
Here the sum is over all curve class $\beta \in H_2(X)_{\geq 0}$. 
$t(q)$ stands for any Laurent polynomial of one variable $q$ with coefficients in $K^0(X)$
\[
 t(q) = \sum_{k\in \mathbb{Z}} \sum_{\alpha=1}^N t_k^{\alpha} \Phi_{\alpha} q^k
\]
with $\{\Phi_{\alpha} \}_{\alpha=1}^N$ a basis in $K^0(X)$. We may rewrite $F^{{K}}_g (t)$ as
\[
F^{{K}}_g(t) = \sum_{n\geq 0}\sum_{\beta}\frac{Q^{\beta}}{n!} \sum_{ \substack{  k_1,\dots,k_n\in \mathbb{Z} \\ \alpha_1,\dots,\alpha_n \in \{1,\dots,N \}  }  }t_{k_1}^{\alpha_1} \cdots t_{k_n}^{\alpha_n} \langle \tau_{k_1}(\Phi_{\alpha_1}), \dots, \tau_{k_n}( \Phi_{\alpha_n} ) \rangle^{X, {K}}_{g,n,\beta}.
\]

\section{Quantum $K$-theory on Calabi--Yau threefolds}

\subsection{GW on CY3} \label{s:3.1}
Before we proceed to quantum $K$-theory for the Calabi--Yau threefolds, we first discuss some relevant statements in Gromov--Witten theory. In the following, we use $\deg_{\mathbb{C}}$ for the Chow degree, i.e., one half of the usual degree in cohomology.

\begin{proposition} \label{p:3.1}
For any Calabi-Yau threefold $X$, if $\deg_{\mathbb{C}} \phi_1 \geq 2$ then
\[
\langle \tau_{k_1}(\phi_1),\dots, \tau_{k_n}(\phi_n) \rangle_{g,n,\beta\neq 0}^{H, \tw} =0,
\]
where $\tw$ denotes cohomological GW invariants with twistings.
\end{proposition}

\begin{proof}
Let 
\[
 \pi_1 : \M_{g,n}(X, \beta) \to \M_{g,1} (X, \beta)
\]
be the forgetful map forgetting the last $n$ marked points and $T\in H^*(\M_{g,n}(X,\beta))$ be the twisting class.
By projection formula
\begin{equation} \label{e:3.1}
\begin{split}
\int_{[\M_{g,n}(X, \beta)]^{\vir}} T \ &\prod_{i=1}^{n} \Big(\psi_i^{k_i} \ev_i^*\phi_i\Big) 
\\
&= \int_{[\M_{g,1}(X, \beta)]^{\vir}} (\ev_1^* \phi_1 ) \ (\pi_1)_*\left( T \ \psi_1^{k_1}\prod_{i=2}^{n} \psi_i^{k_i} \ev_i^*\phi_i \right) .
\end{split}
\end{equation}
Since $[\M_{g,1}(X,\beta)]^{\vir}$ has virtual dimension 1, while $\deg_{\mathbb{C}}(\phi_i) \geq 2$. This completes the proof.
\end{proof}

\begin{proposition} \label{p:3.2}
All descendant Gromov--Witten invariants on a Calabi--Yau threefold $X$ can be reconstructed from $0$-pointed invariants $\{ \langle \cdot \rangle^H_{g,0,\beta} \}_{g, \beta}$.
\end{proposition}

\begin{proof}
This follows from a combination of virtual dimension counts, the string equation, dilaton equation and divisor equation. A closed formula is available in \cite[Proposition~1.6]{Fan_Lee_2019}. 
\end{proof}

\subsection{QK on CY3}
In quantum $K$-theory, things are somewhat different, mostly due to the fact that $K$-theory is more sensitive to the stack structures. For example, let $G$ be a finite group acting on $X$ and $\pi: X \to [X/G]$ the $G$-torsor. In cohomology or Chow, 
\[
 H ([X/G], \mathbb{Q}) = H (X, \mathbb{Q})
\]
and 
\[
 \int_{[X/G]} \alpha = \frac{1}{|G|} \int_X \pi^* \alpha.
\]
However, in $K$-theory,
\[
 K ([X/G]) = K_G (X)
\]
is the $G$-equivariant $K$-theory, which is much richer than $K(X)$. Furthermore, the pushforward of a vector bundle to a point for $X$ is the Euler characteristic of the bundle, while the same operation for $[X/G]$ extracts the $G$-invariant part of the sheaf cohomologies, which is much more intricate.

Nevertheless, one can apply the virtual Hirzebruch--Riemann--Roch theorem on Deligne--Mumford stacks, \cite{Kawasaki_1979, Tonita_RR_2014} which we briefly recall.

Let $M$ be a quasi-smooth DM stack, i.e., a virtual orbifold,  
\[
 IM = \sqcup_i M_i
\]
its inertia stack and $N_i^{\vir}$
the virtual normal bundle. 
Let $\lambda_{-1} (N_i^{\vir})^*$ be the ``$K$-theoretic Euler class'' of $(N_i^{\vir})^*$. For example, if $(N_i^{\vir})^*$ decomposes into a direct sum of line bundles $L_{\alpha}$, then
\[
 \lambda_{-1} (N_i^{\vir})^* = \prod_{\alpha} (1- L_{\alpha}).
\]
Let $g_i$ be the generic automorphism element of $M_i$. For every vector bundle $E$ on $M_i$, $g_i$ acts on $E$. Let $E = \oplus_j E_j$ be an eigenbundle decomposition with eigenvalues $\epsilon_j$ on $E_j$. Denote 
\[
\op{Tr} (E) := \oplus_j \epsilon_j E_j.
\]
Let $F$ be a vector bundle on $M$, and $\tau (F) \in H (IM)$ defined by
\[
 \tau (F) |_{M_i} := \frac{\op{ch} (\op{Tr} F)}{\op{ch} \op{Tr} (\lambda_{-1} (N_i^{\vir})^*)} Td( T^{\vir}_{M_i}) .
\]
The virtual HRR states that
\begin{equation} \label{e:vHRR}
 \chi(M, F) = \sum_i \frac{1}{m_i} \int_{[M_i^{\vir}]}  \tau (F) |_{M_i} ,
\end{equation}
where $m_i$ is the order of the generic automorphism associated to $M_i$. We now prove a $K$-theoretic version of Proposition\ref{p:3.1}.

\begin{proposition}[\cite{Chou_Lee_2023}] \label{p:3.3}
Let $E_1 \in K(X)$ be any element such that $\deg_{\mathbb{C}} \op{ch} (E_1) \ge 2$. Then 
\[
 \langle \tau_{k_1} (E_1) , \ldots , \tau_{k_n} (E_n) \rangle^{K}_{g,n,\beta} =0.
\]
\end{proposition}

\begin{proof}
The proof is a combination of the virtual orbifold Hirzebruch--Riemann--Roch theorem \eqref{e:vHRR} and the arguments in the proof of Proposition~\ref{p:3.1}, including the virtual dimension count and the projection formula. Here two key facts are used. First, $\ev^*E_1$ is pulled back from a scheme and
\begin{equation} \label{e:3.3}
 \op{Tr} (\ev^*E_1) =\ev^* E_1.
\end{equation}
Second, the virtual dimension of any non-identity component of the intertial stack is no greater than the virtual dimension of the identity component. More precisely,
\[
\begin{split}
 &\chi \left( \M_{g,n}(X, \beta), \prod_{j=1}^n L_j^{\otimes k_j} \ev_j^* E_j \right) \\
 = & \sum_i \frac{1}{m_i} \int_{[(M_n)_i^{\vir}]}  \tau (\prod_{j=1}^n L_j^{\otimes k_j} \ev_j^* E_j)  \\
 = & \sum_i \frac{1}{m_i} \int_{[(M_1)_i]^{\vir}} \op{ch}(\ev_1^* E_1 ) (\pi_1)_*\left( \tau ( L_1^{k_1} \prod_{j=2}^{n} L_j^{k_j} \ev_j^* E_j ) \right) \\
 = & 0,
\end{split}
\]
where we have used the fact that the virtual dimensions of all $(M_1)_i$ are less than or equal to $1$ and $\deg_{\mathbb{C}} \op{ch}(\ev_1^* E_1 ) \geq 2$.
\end{proof}

At this moment, we do not have a general result in $K$-theory corresponding to Proposition~\ref{p:3.2}. Nevertheless, in genus zero there is a reconstruction theorem from the set of all one-pointed descendants (i.e., including the cotangent line bundles) to all quantum $K$-invariants with descendants, in the spirit of the reconstruction theorem in \cite{Lee_Pandharipande_2004}. 
The small $J$-function is a generating function of one-pointed descendants
\[
 J (q, Q)\ :=  (1 -q) +  \sum_{\alpha} \sum_{\beta \neq 0}  \Phi_{\alpha} \langle \frac{\Phi^{\alpha}}{1-qL} \rangle_{0,1,\beta} Q^{\beta} 
\]
where $\{ \Phi_{\alpha}, \Phi^{\alpha} \}$ are dual classes with respect to the $K$-theoretic Poincar\'e pairing
\[
 (  \Phi_{\alpha} , \Phi^{\alpha'} ) := \chi ( X, \Phi_{\alpha} \otimes \Phi^{\alpha'} ) = \delta_{\alpha}^{\alpha'}.
\]

\begin{proposition}[\cite{Chou_Lee_2023}] \label{p:3.5}
The small $J$-function in quantum $K$-theory determines all genus zero (descendant) quantum $K$-invariants for Calabi--Yau threefolds.
\end{proposition}

\begin{proof}
This can be shown by a combination the string equation in the quantum $K$-theory \cite{Lee_2004}, a reconstruction theorem in quantum $K$-theory \cite[(22) (23)]{Lee_2004}, a version of Riemann--Roch for virtually smooth stacks, together with virtual dimension counts.

By the string equation, if any insertion is $1 = \O$, then it can be reduced to $K$-invariants of fewer points.  
One can assume that there is no insertion of $1$ by induction. We can therefore assume that the insertions at all marked point look like
\[
  (\ev_i^*(E_i) - r_i 1) 
\]
where 
$E_i$ are vector bundles of rank $r_i$. 
The Chern character of the $i$-th insertion starts at $\deg_{\mathbb{C}} \geq 1$, otherwise the insertion is $1$, contradictory to the assumption.

One can now apply the Hirzebruch--Riemann--Roch for the virtually smooth stacks. Since the $K$-classes pulled back by evaluations are acted trivially by automorphisms, we conclude that only $\prod_i c_1 ( \ev_i^*(E_i) )$ contribute to stacky HRR by virtual dimension counting. Therefore, the quantum $K$-invariants remain unchanged if we replace $(\ev_i^*(E_i) - r_i 1)$ by the corresponding \emph{line bundles} $(\ev_i^*(\bigwedge^{\op{top}} E_i) - 1)$. That is, one can assume that the insertions are all linear combinations of line bundles.

Then the reconstruction theorem in quantum $K$-theory in \cite{Lee_Pandharipande_2004} applies: any descendant quantum $K$-invariants of insertions by line bundles can be reconstructed from the small $J$-function in the quantum $K$-theory. More precisely, all cotangent lines can be moved to the first marked points by induction. Assuming that there are no cotangent line bundles at any but the first marked point, we may further assume that there are no $1$'s at all other points after applying the string equation. Therefore the quantum $K$-invariants look like
\[
 \ev_1^* (E) \otimes L_1^k \otimes \prod_{j=2}^n (\ev_j^*(E_j) - r_j 1),
\]
where $E$ can be arbitrary, e.g., $1$. 
By Proposition~\ref{p:3.3}, $\ev_1^* (E)$ can only contribute through its $\op{ch}_0$ or $\op{ch}_1$ to the virtual HRR formula. $\op{ch}_0$ can be absorbed by the string equation and 
\[
 c_1 (\ev_i^*(E_i) - r_i 1) = c_1 (\bigwedge^{\op{top}} E_i) - 1).
\]
Thus, vector bundle can be replaced by line bundles without changing the quantum $K$-invariants. The reconstruction theorem in \cite{Lee_Pandharipande_2004} applies. 
\end{proof}

\section{Multiple cover formula and the JMGS conjecture}

\subsection{GV $=$ QK in genus zero}
In \cite{Jockers_Mayr_2019} and \cite{Garoufalidis_Scheidegger_2022} H.~Jockers, P.~Mayr and S.~Garoufalidis, E.~Scheidegger formulate a conjectural relation between the Gopakumar--Vafa invariants and quantum $K$-invariants for the \emph{quintic threefold} $X$.
The conjecture is formulated in terms of \emph{small $J$-functions}, a generating function in quantum $K$-theory as well as quantum cohomology.
We generalize their conjecture to all Calabi--Yau threefolds, which should be implicitly in the original proposal by Jockers and Mayr.

We fix the following notation:
\[
\{
\Phi_{\alpha}
\}_{\alpha=1}^N =  \bigsqcup_{i=0}^3 \{ \Phi_{ij}\}_{j=1}^{n_i},
\]
where $\{ \ch ( \Phi_{ij} )\}_{j=1}^{n_i}$ forms a basis in $H^{2i}(X)$. In particular, 
\[
 \{ \Phi_{0j} \}_{j=1}^{n_0} = \{ \Phi_{01} =\mathcal{O} \}.
\]
Let $\{ \Phi^{ij} \}$ be the dual basis of $\{ \Phi_{ij} \}$ with respect to the K-theoretic Poincar\'e pairing:
\[
(\Phi_a, \Phi_b)^{K} : = \chi (X, \Phi_a \Phi_b).
\]
The following fact will be used later.
\begin{lemma} For $i = 0$ and $1$,
\[ 
\ch( \Phi^{ij} ) \in H^{3-i}(X).
\]
For $i=2$ and $3$,
\[
\ch ( \Phi^{ij}) \in H^{ \geq 3-i} (X).
\]
\end{lemma}
\begin{proof}
$\Phi^{ij}$ can be written as:
\[
\Phi^{ij}:= \ch^{-1} \Big(  \td(TX)^{-1} PD (\ch (\Phi_{ij})) \Big), 
\]
where PD denotes the Poincar\'e dual. The lemma follows from the definition of $\Phi_{ij}$ and that $\td(TX)^{-1} \in 1 + H^{\geq 2}(X)$.
\end{proof}


We now generalize the JMGS conjecture to general Calabi--Yau threefolds.

\begin{conjecture} [{cf.~\cite{Jockers_Mayr_2019, Garoufalidis_Scheidegger_2022}}]  
\label{conjectureJK}
\[
\begin{split}
 \frac{1}{1-q} \left[ J(q, Q)\right] := & \frac{1}{1-q} \left[  (1 -q) +  \sum_{\alpha} \sum_{\beta \neq 0}  \Phi_{\alpha} \langle \frac{\Phi^{\alpha}}{1-qL} \rangle_{0,1,\beta} Q^{\beta} \right]
\\
 = & 1+ \sum_{\vec{d} \in H_2(X,\mathbb{Z})} \sum_{r=1}^{\infty}  \bigg[ \sum_{j=1}^{n_1} \Phi^{1j} \Big(\int_{\vec{d}} \ch ( \Phi_{1j}) \Big) \, a(r, q^r) \, \GV_{\vec{d}} \, Q^{r\vec{d}} \\
 & \qquad \qquad \qquad  + \Phi^{01}  b(r, q^r) \, \GV_{\vec{d}} \, Q^{r \vec{d}} \bigg], 
\end{split}
\]
where
\begin{equation} \label{e:4.1}
\begin{split}
     a(r,q^r) &= \frac{(r-1)}{1-q^r} + \frac{1}{(1-q^r)^2},
    \\
     b(r,q^r) &= \frac{r^2-1}{1-q^r} + \frac{3}{(1-q^r)^2} -\frac{2}{(1-q^r)^3}.
\end{split}
\end{equation}
\end{conjecture}

The main result of Part II is a proof of this conjecture. 
\begin{theorem}[\cite{Chou_Lee_2023}] \label{t:4.3}
    Conjecture ~\ref{conjectureJK} holds.
\end{theorem}

\begin{remark}
By Proposition~\ref{p:3.5}, Theorem~\ref{t:4.3} implies that the quantum $K$-theory and Gopakumar--Vafa theory are equivalent for all Calabi--Yau threefolds in genus zero.
\end{remark}

\subsection{Multiple cover formula}

\begin{lemma} \label{l:4.4}
For the total space $X_{-1,-1}$ of $\O(-1) \oplus \O (-1) \to P^1$, the Gopakumar--Vafa invariants $\op{GV}_{0,1} =1$ (genus zero and degree $1$) and $\op{GV}_{g,d} =0$ otherwise
\end{lemma}

\begin{proof}
This follows from the Gopakumar--Vafa equation \eqref{e:1.1'}, and the results in Gromov--Witten theory: for $g=0$ the Voisin--Aspinwall--Morrison formula
\[
 \op{GW}_{0,0,d} = \frac{1}{d^3},
\]
for $g=1$, the BCOV and Graber--Pandharipande formula
\[
 \op{GW}_{1,0,d} = \frac{1}{12 d},
\]
and for $g \geq 2$ the Faber--Pandharipande formula
\[
 \op{GW}_{g,0,d} = \frac{|B_{2g}| d^{2g-3}}{2g \cdot (2g-2)!}.
\]
\end{proof}

In order to consider the small $J$-function of $X_{-1,-1}$ in quantum $K$-theory, we consider its compatification by the ``infinity divisor''. Let 
\[
 Y_{-1,-1} := P_{P^1}(\mathcal{O}(-1)\oplus \mathcal{O}(-1)\oplus \mathcal{O} ).
 \]
Let $P = \pi^*\mathcal{O}(-1)$ with $\pi: Y_{-1,-1} \rightarrow P^1$, and $t = \mathcal{O}(D_{\infty})$ with $D_{\infty}\subset Y_{-1,-1}$, the infinity divisor. We denote $Q^r := Q^{r\ell}$, where 
\[
 \ell := [P^1] \xhookrightarrow{0} Y_{-1,-1}
\] 
is the line class in the ``zero section'' of the projective bundle. In the following, we consider the \emph{specialized} small $I$-function and $J$-function of $Y_{-1,-1}$ only for the curve classes in the zero section, i.e., multiples of $\ell$.

\begin{lemma} \label{l:4.5} 
The small $I$-function for 
$Y_{-1,-1}$, with curve classes in the zero section, is
\[
  I^{Y_{-1,-1}}(q, Q) = (1-q) \left[ 1+ \sum_{r=1}^{\infty} Q^r \frac{(1-Pt)^2 \prod_{m=1}^{r-1}(1-Pt q^m)^2}
  {(Pt)^{2r}q^{r(r-1)} \prod_{m=1}^r (1-Pq^m)^2} \right].
\]
\end{lemma}

\begin{proof}
This lemma follows from the computations of the $I$-functions for toric manifolds via fixed point localization by A.~Givental and collaborators. See \cite{Givental_Tonita_2011, Givental_PEK}. Their formula for small $I$-function gives
\[
  I^{Y_{-1,-1}}(q, Q) = (1-q) \left[ 1+ \sum_{r=1}^{\infty} Q^r \frac{\prod_{m={-r+1}}^{0}(1-P^{-1} t^{-1} q^m)^2}
  {\prod_{m=1}^r (1-Pq^m)^2} \right].
\]
A simple manipulation gives the above presentation.
\end{proof}

We note that in this case the small $I$-function includes a factor of $(1-Pt)^2$, the $K$-theoretic normal bundle of $P^1$ embedded in $Y_{-1,-1}$. One may think of this as an $I$-function of a toric completion $P_{P^1} (\O(-1) \oplus \O (-1) \oplus \O)$, with curve class in the base $P^1$. 
This small $I$-function is different from the small $J$-function, as it has poles at $q=0$. 
(However, it satisfies $\lim_{q \to \infty} I^{X_{-1,-1}}(q) =0$.) Using this $I$-function one can obtain the small $J$-function via the ``\emph{generalized} mirror transform'' (also known as the explicit reconstruction, or Birkhoff factorization).

\begin{proposition} \label{p:4.6}
The small $J$-function for 
$Y_{-1,-1}$, with curve classes in the zero section, is 
\begin{equation} \label{e:4.2}
\begin{split}
\frac{1}{1-q} & J^{Y_{-1,-1}}(q, Q) = 1 + (1-Pt)^2\Big(1 +(1-P)\Big)\sum_{r \geq 1} \, Q^r a(r, q^r) 
\\
& + (1-Pt)^2(1-P) \sum_{r\geq 1} Q^r \, b(r, q^r) , 
\end{split} 
\end{equation}
where $a(r, q^r), b(r, q^r)$ are defined in \eqref{e:4.1}.
\end{proposition}
\begin{proof} 
The $K$-ring of $Y_{-1,-1}$ has the following presentation
\[
 K(Y_{-1,-1}) = \frac{\mathbb{Z} [P, t]}{\left( (1-P)^2, (1-Pt)^2(1-t) \right)}.
\]
We use the relations of $K(Y_{-1,-1})$ to rewrite $I^{Y_{-1,-1}}$ as follows. Since $(1-Pt)^2(1-t) =0$, we have
\begin{equation} \label{e:4.3}
\begin{split}
   &(1-Pt)^2 (1-Pq^m t) = (1-Pt)^2 (1-Pq^m [1-(1-t)]) \\
   = &(1-Pt)^2 (1-Pq^m).  
\end{split}
\end{equation}
Therefore,
\[
\begin{split}
    & I^{Y_{-1,-1}} - (1-q)\\
    = &(1-q) \sum_{r=1}^{\infty} Q^r \frac{(1-Pt)^2} 
    {(Pt)^{2r}q^{r(r-1)}(1-Pq^r)^2}
    \\
    = &(1-q)\sum_{r\geq 1}Q^r\left[ \Big( \sum_{i=1}^{r-1} \frac{i}{q^{r(r-i)}}\Big)(1-Pt)^2  + \Big( \sum_{i=1}^{r-1} \frac{i(2r-i+1)}{q^{r(r-i)}} \Big) (1-Pt)^3 \right]
    \\
    & + (1-q) \Big((1-Pt)^2 +(1-Pt)^3\Big)\sum_{r\geq 1} Q^r \Big( \frac{r-1}{1-q^r} + \frac{1}{(1-q^r)^2} \Big)
    \\
    & + (1-q)(1-Pt)^3 \sum_{r\geq 1}Q^r\Big( \frac{r^2-1}{1-q^r} + \frac{3}{(1-q^r)^2} -\frac{2}{(1-q^r)^3} \Big).
\end{split}
\]
In the first equality, the factor $\prod_{m=1}^{r-1} (1-Ptq^m)^2$ in the numerator and $\prod_{m=1}^{r-1}(1-Pq^m)^2$ in the denominator cancel each other due to the presence of $(1-Pt)^2$ and \eqref{e:4.3}.
The second equality follows from an explicit computation.

The first line after the second equal sign lies in $\mathcal{K}_+$ and the rest lies in $\mathcal{K}_-$. Consider the reconstruction theorem \cite[Theorem~2]{Givental_PEVIII_ER}
\[
\begin{split}
    &J^{X_{-1,-1}}(q, Q) =  \sum_{d\geq 0}  I_d^{X_{-1,-1}} Q^d \cdot 
    \\ 
    & \quad  \exp \left( \Big(\sum_{k>0}\frac{ \Psi^k}{k} \Big) \Big(\frac{ \delta(Q)(1-Pq^{d}) + \sum_{i=0}^3 \e_i(Q) (1-Pt q^{d})^i}{(1-q)}\Big) \right) \cdot
    \\
    &\qquad \left( s(q,Q)(1-Pq^d) + \sum_{i=0}^3 r_i(q,Q) (1-Ptq^d)^i \right),
\end{split}
\]
for some uniquely determined $\e_i(Q)$, $\delta(Q)$, $s(q,Q)$ and $r_i (q, Q)$, where
\[
\begin{split}
    \e_i(Q)& = \sum_{j \geq 1} \e_{ij}Q^j \in \mathbb{Q}[\![Q]\!], \\
    \delta(Q) &=\sum_{j \geq 1} \delta_{j}Q^j \in \mathbb{Q}[\![Q]\!],
    \\
    r_i(q,Q) 
   & = \sum_{j\geq 0} r_{ij}(q) Q^j \in \mathbb{Q}[q,q^{-1}][\![Q]\!],
   \\
   s(q,Q) 
   & = \sum_{j\geq 0} s_j(q) Q^j \in \mathbb{Q}[q,q^{-1}][\![Q]\!].
\end{split}
\]
A direct computation by induction on the degree of the Novikov variable shows that 
\[
r_0(q,Q) = 1, \quad \e_1(Q)=\e_2(Q)=\e_3(Q)= \delta(Q)=r_1(q,Q)=s(q,Q)=0.
\]
and that $r_2(q,Q)$ and $r_3(q,Q)$ will not change the $\mathcal{K}_-$ part. 
This concludes the proof.
\end{proof}

\begin{corollary}[Multiple cover formula \cite{Chou_Lee_2023}] \label{c:4.8}
 The small $J$-function for $X_{-1,-1}$ is
 \[
 \begin{split}
\frac{1}{1-q} & J^{X_{-1,-1}}(q) = 1 + \Big(1 +(1-P)\Big)\sum_{r \geq 1} \, Q^r a(r, q^r) 
\\
& + (1-P) \sum_{r\geq 1} Q^r \, b(r, q^r) , 
\end{split}
 \]
where $a(r, q^r), b(r, q^r)$ are defined in \eqref{e:4.1}.
\end{corollary}

\begin{proof}
Since the zero section $P^1$ has normal bundle $\O(-1) \oplus \O(-1)$, the quantum $K$-invariants of $r \ell$ in $X_{-1,-1}$ are exactly the same as those in $Y_{-1,-1}$. The only difference in $J$-functions comes from different bases of the $K$-groups and the Poincar\'e pairing. The net result is the removal of the factor $(1-Pt)^2$ from the specialized $J^{Y_{-1,-1}}(q, Q)$ for non-zero degree terms.
\end{proof}


\subsection{Virtual Clemens' conjecture}

We now give a heuristic derivation of Conjecture~\ref{conjectureJK} and a heuristic interpretation of the relationship between Gopakumar--Vafa invariants and quantum $K$-invariants at genus zero by a multiple cover formula. This has served to guide us in our search for the current formulation of Conjecture~\ref{conjectureJK}, even though the actually proof follows a completely different approach. Of course, the original formulations of Jockers--Mayr \cite{Jockers_Mayr_2019} and Garoufalidis--Scheidegger \cite{Garoufalidis_Scheidegger_2022} have been enormous help.

Assume that we are given an ``ideal'' Calabi--Yau threefold $X$ satisfying a ``virtual Clemens' conjecture''. That is, there are, up to deformations, finitely many isolated rational curves $\{C_i \}$. Furthermore, they are all smooth $(-1,-1)$ curves. 

By Lemma~\ref{l:4.4}, each isolated $(-1,-1)$-curve (in any degree $\vec{d}$) contributes $1$ to the Gopakumar--Vafa invariants, independently of $\vec{d}$. Therefore, there are $\GV_{0,\vec{d}}$ isolated $(-1,-1)$-curves in degree $\vec{d}$. For each of these isolated curves, quantum $K$-theory allows multiple $r$-covers of the isolated $(-1,-1)$-curve. The coefficients $a(r, q^r)$ and $b(r,q^r)$ of the $r$-covers come from the $J$-function of $X_{-1,-1}$. The only addition is the factor of $\int_{\vec{d}} D_j$, where the divisor $D_j = \op{ch} (\Phi^{1j})$ comes from the divisor axiom. 

In summary, the ``virtual Clemens conjecture'' implies that GV $=$ QK for all Calabi--Yau threefolds in genus zero via the multiple cover contributions.

This line of thoughts lead us to believe that, in order to generalize this to higher genera, the most important ingredient is the higher genus multiple cover formula in quantum $K$-theory. It is entirely possible that the higher genus multiple cover formulas will serve as universal coefficients, similar to the genus zero case. We intend to pursue this in future works.




\bibliographystyle{plain}
    
\bibliography{zbib}

\begin{thebibliography}{10}

\bibitem{Chou_Lee_2022}
You-Cheng Chou and Y.-P. Lee.
\newblock Quantum {$K$}-theory and {G}opakumar--{V}afa invariants {I}. {T}he
  quintic threefolds.
\newblock {\em Preprint}, page 40pp, 2022.

\bibitem{Chou_Lee_2023}
You-Cheng Chou and Y.-P. Lee.
\newblock Quantum {$K$}-theory and {G}opakumar--{V}afa invariants {II}.
  {C}alabi--{Y}au threefolds.
\newblock {\em Preprint}, page in preparation, 2023.

\bibitem{Fan_Lee_2019}
Honglu Fan and Yuan-Pin Lee.
\newblock Towards a quantum {L}efschetz hyperplane theorem in all genera.
\newblock {\em Geom. Topol.}, 23(1):493--512, 2019.

\bibitem{Garoufalidis_Scheidegger_2022}
Stavros Garoufalidis and Emanuel Scheidegger.
\newblock On the quantum {K}-theory of the quintic.
\newblock {\em SIGMA Symmetry Integrability Geom. Methods Appl.}, 18:Paper No.
  021, 2022.

\bibitem{Givental_2000}
Alexander Givental.
\newblock On the {WDVV} equation in quantum {$K$}-theory.
\newblock volume~48, pages 295--304. 2000.
\newblock Dedicated to William Fulton on the occasion of his 60th birthday.

\bibitem{Givental_PEVIII_ER}
Alexander Givental.
\newblock Permutation-equivariant quantum quantum {$K$}-theory {VIII}.
  {E}xplicit reconstruction.
\newblock {\em arXiv: Algebraic Geometry}, 2015.

\bibitem{Givental_PEK}
Alexander Givental.
\newblock Permutation-equivariant quantum quantum {$K$}-theory {I}-{XI}.
\newblock {\em arXiv: Algebraic Geometry}, 2015-2017.

\bibitem{Givental_Lee_2003}
Alexander Givental and Yuan-Pin Lee.
\newblock Quantum {$K$}-theory on flag manifolds, finite-difference {T}oda
  lattices and quantum groups.
\newblock {\em Invent. Math.}, 151(1):193--219, 2003.

\bibitem{Givental_Tonita_2011}
Alexander Givental and Valentin Tonita.
\newblock The {H}irzebruch-{R}iemann-{R}och theorem in true genus-0 quantum
  {K}-theory.
\newblock In {\em Symplectic, {P}oisson, and noncommutative geometry},
  volume~62 of {\em Math. Sci. Res. Inst. Publ.}, pages 43--91. Cambridge Univ.
  Press, New York, 2014.

\bibitem{Gopakumar_Vafa_1}
Rajesh Gopakumar and Cumrun Vafa.
\newblock {M}-theory and topological strings--{I}.
\newblock {\em arXiv:hep-th/9809187}, pages 1--14, 12 1998.

\bibitem{Gopakumar_Vafa_2}
Rajesh Gopakumar and Cumrun Vafa.
\newblock {M}-theory and topological strings--{II}.
\newblock {\em arXiv:hep-th/9812127}, pages 1--19, 12 1998.

\bibitem{Ionel_Parker_2018}
Eleny-Nicoleta Ionel and Thomas~H. Parker.
\newblock The {G}opakumar-{V}afa formula for symplectic manifolds.
\newblock {\em Ann. of Math. (2)}, 187(1):1--64, 2018.

\bibitem{Jockers_Mayr_2019}
Hans Jockers and Peter Mayr.
\newblock Quantum {K}-theory of {C}alabi-{Y}au manifolds.
\newblock {\em J. High Energy Phys.}, (11):011, 20, 2019.

\bibitem{Kawasaki_1979}
Tetsuro Kawasaki.
\newblock The {R}iemann-{R}och theorem for complex {$V$}-manifolds.
\newblock {\em Osaka Math. J.}, 16(1):151--159, 1979.

\bibitem{Lee_2004}
Y.-P. Lee.
\newblock Quantum {$K$}-theory. {I}. {F}oundations.
\newblock {\em ArXiV.org math.AG/0105014, Duke Math. J.}, 121(3):389--424,
  2004.

\bibitem{Lee_Pandharipande_2004}
Y.-P. Lee and R.~Pandharipande.
\newblock A reconstruction theorem in quantum cohomology and quantum
  {$K$}-theory.
\newblock {\em Amer. J. Math.}, 126(6):1367--1379, 2004.

\bibitem{Maulik_Toda_2018}
Davesh Maulik and Yukinobu Toda.
\newblock Gopakumar-{V}afa invariants via vanishing cycles.
\newblock {\em Invent. Math.}, 213(3):1017--1097, 2018.

\bibitem{Tonita_RR_2014}
Valentin Tonita.
\newblock A virtual {K}awasaki-{R}iemann-{R}och formula.
\newblock {\em Pacific J. Math.}, 268(1):249--255, 2014.

\end{thebibliography}
    
\end{document}